\documentclass[12pt,leqno,twoside]{amsart}

\usepackage{amssymb,amsmath,amsthm,soul,color}

\usepackage[noadjust]{cite}

\usepackage{t1enc}

\usepackage[utf8]{inputenc}
\usepackage{indentfirst}

\usepackage{todonotes}

\usepackage[left=2.5cm, right=2.5cm, top=2cm, bottom=2cm]{geometry}

\usepackage{url}

\usepackage{mathrsfs}

\usepackage{enumitem,graphicx}
\usepackage[colorlinks=true,urlcolor=blue,
citecolor=red,linkcolor=blue,linktocpage,pdfpagelabels,
bookmarksnumbered,bookmarksopen]{hyperref}
\usepackage[metapost]{mfpic}
\usepackage[normalem]{ulem}

\usepackage{mathtools}

\usepackage{empheq}

\newtheorem{Th}{Theorem}[section]
\newtheorem{Prop}[Th]{Proposition}
\newtheorem{Lem}[Th]{Lemma}
\newtheorem{Cor}[Th]{Corollary}

\newtheorem{Rem}[Th]{Remark}

\newcommand{\vp}{\varphi}

\newcommand{\R}{\mathbb{R}}

\newcommand{\N}{\mathbb{N}}

\newcommand{\cC}{{\mathcal C}}
\newcommand{\cD}{{\mathcal D}}

\newcommand{\cM}{{\mathcal M}}

\newcommand{\cO}{{\mathcal O}}

\newcommand{\cS}{{\mathcal S}}

\newcommand{\tu}{\widetilde{u}}

\numberwithin{equation}{section}

\newcommand{\rn}{\R^N}

\newcommand{\codim}{\mathrm{co}\,\mathrm{dim}\,}

\begin{document}


\title{Normalized solutions to polyharmonic equations with Hardy-type potentials and exponential critical nonlinearities}

\author[B. Bieganowski]{Bartosz Bieganowski}
	\address[B. Bieganowski]{\newline\indent
			Faculty of Mathematics, Informatics and Mechanics, \newline\indent
			University of Warsaw, \newline\indent
			ul. Banacha 2, 02-097 Warsaw, Poland}	
			\email{\href{mailto:bartoszb@mimuw.edu.pl}{bartoszb@mimuw.edu.pl}}	
			
\author[O. H. Miyagaki]{Ol\'impio Hiroshi Miyagaki}
\address[O. H. Miyagaki]{\newline\indent Department of Mathematics
	\newline\indent 
Federal University of S\~ao Carlos-UFSCAR
\newline\indent 
Rod. Washington Luis Km 235, 13565-905 S\~ao Carlos-SP, Brazil}
\email{\href{mailto:olimpio@ufscar.br}{olimpio@ufscar.br}}

\author[J. Schino]{Jacopo Schino}
\address[J. Schino]{\newline\indent
	Faculty of Mathematics, Informatics and Mechanics
	\newline\indent 
	University of Warsaw
	\newline\indent 
	ul. Banacha 2, 02-097 Warsaw, Poland}
\email{\href{mailto:j.schino2@uw.edu.pl}{j.schino2@uw.edu.pl}}		
	
	
	\maketitle
	
	\pagestyle{myheadings} \markboth{\underline{B. Bieganowski, O. H. Miyagaki, J. Schino}}{
		\underline{Normalized solutions to polyharmonic singular equations with exponential critical growth}}

\begin{abstract} 
Via a constrained minimization, we find a solution $(\lambda,u)$ to the problem
\begin{equation*}
\begin{cases}
(-\Delta)^m u+\frac{\mu}{|x|^{2m}}u + \lambda u = \eta u^3 + g(u)\\
\int_{\mathbb{R}^{2m}} u^2 \, dx = \rho
\end{cases}
\end{equation*}
with $1 \le m \in \mathbb{N}$, $\mu,\eta \ge 0$, $\rho > 0$, and $g$ having exponential critical growth at infinity and mass supercritical growth at zero.

\medskip

   
\end{abstract}

\maketitle

\section{Introduction}

In this paper, we search for solutions $(\lambda,u)$ to
\begin{equation}\label{eq}
\begin{cases}
\displaystyle (-\Delta)^m u + \frac{\mu}{|x|^{2m}}u + \lambda u = \eta u^3 + g(u), \quad x \in \R^{2m} \\
\displaystyle \int_{\R^{2m}} u^2 \, dx = \rho,
\end{cases}
\end{equation}
where $1 \le m \in \N$, $\mu,\eta \ge 0$,
$\rho > 0$, and $g \colon \R\to\R$ is a nonlinear term that satisfies suitable assumptions -- cf. (\ref{A0})--(\ref{A4}) below. In particular, the \textit{mass} $\rho$ is prescribed, $\lambda \in \R$ is unknown and will arise as a Lagrange multiplier, and $u \colon \R^{2m} \to \R$ belongs to an appropriate Sobolev space. Solutions to problems like \eqref{eq}, i.e., where the mass is given a priori, are usually referred to as \textit{normalized solutions.}

System \eqref{eq} appears when one studies the Cauchy problem for the time-dependent \textit{polyharmonic} equation
\begin{equation}\label{schroed-t}
\begin{cases}
\displaystyle \mathbf{i} \partial_t \Psi = (-\Delta)^m \Psi + \frac{\mu}{|x|^{2m}}\Psi  - \mathbf{f}(|\Psi|)\Psi, \\
\Psi(0, \cdot) = u \in L^2 (\R^{2m}) \setminus \{ 0\},
\end{cases}
\end{equation}
which was considered in \cite{IK,Turitsyn} with $m=2$ and $\mu = 0$ to study the stability of solitons in magnetic materials once the effective quasiparticle mass becomes infinite. More precisely, if \textit{stationary-wave solutions} to \eqref{schroed-t} are looked for, i.e., $\Psi(x, t) = e^{i\lambda t} u(x)$, then $|\Psi(t,\cdot)| = |u|$ for all $t \in \R$ and one obtains \eqref{eq}. In fact, if $\mathbf{f}$ satisfies opportune hypotheses, then the $L^2$-norm of \textit{any} solution to \eqref{schroed-t} is constant in time; indeed, multiplying by the complex conjugate of $\Psi$, integrating, and taking the imaginary part leads to $\frac{d}{dt} \int_{\R^N} |\Psi(t,x)|^2 \, dx = 0$. This is particularly meaningful from the physical point of view; for instance, when $m=1$, the mass represents the total number of atoms in Bose--Einstein condensation \cite{Malomed} or the total power in nonlinear optics \cite{Buryak}. It is an interesting mathematical matter, then, to investigates the cases of higher values of $m$ or nonzero $\mu$ (or both).

In general, we focus on mass-critical ($\eta > 0$) or mass-supercritical ($\eta = 0$) problems with
\begin{equation}\label{example}
g(s) = p \beta |s|^{p-2}s e^{\alpha_m s^2}
\end{equation}
as a model in mind\footnote{Concerning the constant $p$ in front of $\beta$, cf. assumption (\ref{A2}) below.}, with $\beta > 0$, $p > 4$, and
\begin{equation}\label{alpha_m}
\alpha_m := \frac{2m (2\pi)^{2m}}{\omega_{2m-1}},
\end{equation}
where $\omega_{2m-1}$ is the surface measure of the unit sphere $\mathbb{S}^{2m-1}$

Our approach is variational, so we will find a solution to \eqref{eq} by seeking critical points of a suitable functional. To do so, we must first define the function space
\begin{equation}\label{def:X}
X^m := \left\{ u \in H^{m}(\R^{2m}) \ : \ \mu \int_{\R^{2m}} \frac{u^2}{|x|^{2m}} \, dx < +\infty \right\}
\end{equation}
and its symmetric subspace
\begin{equation*}
X_\textup{rad}^m := \left\{ u \in H^{m}_\textup{rad}(\R^{2m}) \ : \ \mu \int_{\R^{2m}} \frac{u^2}{|x|^{2m}} \, dx < +\infty \right\},
\end{equation*}
where $H^{m}_\textup{rad}(\R^{2m})$ is the subspace of $H^m(\R^{2m})$ consisting of radially symmetric functions. In particular, when $\mu = 0$, $X^m = H^{m}(\R^{2m})$ and $X_\textup{rad}^m = H^{m}_\textup{rad}(\R^{2m})$. Next, we define the functional $J \colon X^m \to \R$ by
\begin{equation}\label{def:J}
J(u) = \frac12 \int_{\R^{2m}} |\nabla^{m} u|^2 + \frac{\mu}{|x|^{2m}} u^2 \, dx - \int_{\R^{2m}} \frac{\eta}{4} u^4 + G(u) \, dx,
\end{equation}
where $G(s) := \int_0^s g(t) dt$ and
\begin{equation}\label{def:nabla-m}
\nabla^m u :=
\begin{cases}
\Delta^{m/2} u & \text{if } m  \text{ is even,}\\
\nabla \Delta^{(m-1)/2} u & \text{if } m \text{ is odd.}
\end{cases}
\end{equation}
Observe that $J \in \cC^1(X^m)$ in view of the norm in $X^m$, see \eqref{eq:norm}, and Lemma \ref{lem:do}.
We also introduce the following sets:
\begin{align*}
\cS & := \left\{ u \in L^2(\R^{2m})\ : \ \int_{\R^{2m}} |u|^2 \, dx = \rho \right\},\\
\cD & := \left\{ u \in L^2(\R^{2m})\ : \ \int_{\R^{2m}} |u|^2 \, dx \leq \rho \right\},
\end{align*}
and, letting $H(s) := g(s)s - 2G(s)$,
\begin{equation}\label{def:cM}
\cM := \left\{ u \in X^m_\textup{rad} \setminus \{0\} \ : \  \int_{\R^{2m}} |\nabla^{m} u|^2 + \frac{\mu}{|x|^{2m}} u^2 \, dx = \int_{\R^{2m}} \frac{\eta}{2} u^4 + H(u) \, dx \right\}.
\end{equation}

Before listing our assumptions, we recall from \cite{BM} the following notation for given functions $f_1,f_2 \colon \R \to \R$. We write $f_1(s)\preceq f_2(s)$ for $s\in\R\setminus\{0\}$ if and only if $f_1 \leq f_2$ and for every $\gamma>0$ there exist $s_1>0$ and $s_2<0$ such that $|s_j|<\gamma$ and $f_1(s_j)<f_2(s_j)$, $j \in \{1,2\}$.

Finally, we state our assumptions about $g$.
\begin{enumerate}[label=(A\arabic{*}),ref=A\arabic{*}] \setcounter{enumi}{-1}
	\item \label{A0} $g$ and $h:=H'$ are continuous, 
	\[
	|g(s)| + |h(s)| = \cO(|s|) \text{ as } s \to 0, \ \text{and} \ \lim_{|s| \to +\infty} (|g(s)| + |h(s)|) / e^{\alpha s^2} = 0 \text{ for all } \alpha > \alpha_m.
	\]
	\item \label{A1} $\lim_{s\to 0} H(s)/s^{4} = 0$.
	\item \label{A2} There exist $\beta > 0$ and $p > 4$ such that $G(s) \ge \beta |s|^{p}$ for all $s \in \R$.
	\item \label{A3} $4 H(s) \preceq  h(s)s$ for $s\in \R\setminus\{0\}$.
	\item \label{A4} There exists $\theta > 4$ such that $0 \le \theta G(s) \le s g(s)$ for all $s \in \R$.
\end{enumerate}
It is readily seen that \eqref{example} satisfies all the assumptions above. Notice that (\ref{A0}) includes the case where $g$ has \textit{exponential critical growth} at infinity (i.e., when the limit is $+\infty$ if $\alpha < \alpha_m$), but does not exclude the \textit{exponential subcritical} one either (that is, when (\ref{A0}) holds with $0$ instead of $\alpha_m$).

What follows is the main outcome of this paper.

\begin{Th}\label{th:main}
Suppose that (\ref{A0})--(\ref{A4}) are satisfied,
\begin{equation}\label{eq:Hstrict}
\eta C_4^4 \rho < 2,
\end{equation}
and
\begin{equation}\label{eq:beta}
\beta > \left(\frac{(\theta - 2)(p - 4)}{(\theta - 4)(p - 2)}\right)^\frac{p-4}{2} \left(1 - \frac{\eta}{2} C_4^4 \rho\right) \frac{1}{(p-2) C_p^p \rho},
\end{equation}
where $C_4$ and $C_p$ are defined in \eqref{eq:GN}. Then, there exist $\lambda>0$ and $u \in \cS \cap \cM$ such that
$$
J(u) = \inf_{\cD \cap \cM} J = \inf_{\cS \cap \cM} J > 0
$$
and $(\lambda,u)$ is a solution  of \eqref{eq}. If, in addition, $m=1$ and $g$ is odd, then $u \ge 0$.
\end{Th}

\begin{Rem}
(i) When $m=1$, a standard argument proves that the function $u$ given by Theorem \ref{th:main} is continuous in $\R^{2m} \setminus \{0\}$ (in $\R^{2m}$ if $\mu = 0$). Then, if $g$ is odd, the maximum principle yields that $u(x) > 0$ for all $x \ne 0$ (all $x \in \R^{2m}$ if $\mu = 0$).\\
(ii) Concerning the conditions introduced in Theorem \ref{th:main}, \eqref{eq:Hstrict} is used in several steps of the proof to ensure that $\eta |u|_4^4$ is suitably controlled by $|\nabla^m u|_2^2$ if $u \in \cD$, while \eqref{eq:beta} is used in Lemma \ref{lem:c_0attained} to make sure the compactness result in Lemma \ref{lem:cpt} can be applied.
\end{Rem}

Let us recall some relevant literature. Since a lot of work has been carried out about normalized solutions, we limit ourselves to the contexts most similar to \eqref{eq}, that is, euclidean setting and at least one of these aspects: homogeneous polyharmonic operators, potentials, and exponential critical growth.

In \cite{LZ,Ma_Chang}, the authors find a normalized solution to an \textit{autonomous} and \textit{biharmonic} (i.e., $m=2$) equation where the nonlinear term has Sobolev-critical growth at infinity. In \cite{Phan}, instead, the author obtains existence, nonexistence, and blow-up results for a purely mass-critical biharmonic equation with a potential that is either nonnegative and coercive or with its negative part lying in a suited sum of Lebesgue spaces. Moving on, biharmonic equations with a nonautonomous nonlinearity are investigated in \cite{Wang3,ZLG}, where multiple solutions are found in a mass-subcritical regime, and \cite{LuZhang}, where a mountain-pass solution is obtained in a mass-supercritical and Sobolev subcritical framework. Concerning higher values of $m$, to the best of our knowledge, the only pieces of work that debate normalized solutions (to autonomous or nonautonomous equations) are \cite{ScSm}, about one-dimensional problems, and \cite{BMS}, where a solution is obtained as a constrained minimizer in dimension $N \ge 2m$, and nonexistence results are proved as well; additionally, \cite{BMS} contains a Poho\v{z}aev-type result for solutions to polyharmonic equations with Hardy-like potentials that is crucial for our approach.

The literature is richer when concerning Schr\"odinger equations (i.e., $m=1$) with potentials; however, the assumptions about them are usually not satisfied by $\mu |x|^{-2}$, $\mu > 0$. We mention \cite{GHLY,IM,PPVV,ZZ} for (at least) continuous potentials $V \colon \R^{2m} \to \R$, \cite{Noris,yLZ} for coercive ones, \cite{AJ1,TZZZ} for locally essentially bounded potentials with nonempty interior of the pre-image of $0$, and \cite{AJ1,BartschMolle,Ding,IY,yLZ,Molle} for contexts where the functional $u \mapsto \int_{\R^N} V(x) u^2 \, dx$ is well defined in $H^1(\R^N)$; in addition, smallness assumptions about $V$ are often required. Except for the already mentioned \cite{BMS}, the only article dealing with a Hardy potential (for $m=1$) seems to be \cite{hLwZ} (we also mention \cite{LW}, where, in dimension $3$, the vector $x$ is replaced with its first two components).

Finally, we recall the past work about normalized solutions to autonomous Schr\"odinger equations with exponential critical nonlinearities. In \cite{AJM1}, the authors consider a nonlinear term with mass-supercritical growth at the origin and find a solution under the assumption, among others, that the mass is less than $1$.
This restriction on the mass is waived in \cite{CLY}, where the nonlinearity has a mass-supercritical growth at the origin, in \cite{LiZhang,MW}, where such a growth is allowed to be mass-critical, and in \cite{CT}, where the right-hand side of the equation consists of a superlinear power plus an explicit exponential-type function. We also mention \cite{CRTY}, where the nonlinear term also depends on $x$ but there is no potential in the equation.

To find a solution to \eqref{eq} with the lowest possible energy among radial solutions, we adopt a strategy introduced in \cite{BM}, refined in \cite{BMS,MS}, and adapted to different scenarios in \cite{CLY,yLZ,LRZ}. Briefly speaking, since every nontrivial solution to \eqref{eq} belongs to $\cM$ (see Section \ref{sec:setting}), we minimize the functional $J$ over the set $\cD \cap \cM$. In this process, the various parameters in (\ref{A0})--(\ref{A4}) play an important role, see Section \ref{sec:proofs}.

In the next sections, the symbol $\lesssim$ stands for the inequality $\le$ up to a multiplicative positive constant whose precise value is irrelevant, and $t'$ is the H\"older conjugate of $t \in [1,+\infty]$.

\section{Preliminaries}\label{sec:setting}

We endow the space $H^m (\R^{2m})$ with the norm
$$
\| u\|_{H^m}^2 := |\nabla^m u|_2^2 + |u|_2^2,
$$
equivalent to the standard one, where $|\cdot|_k$ stands for the usual $L^k$-norm and $\nabla^m$ is given by \eqref{def:nabla-m}. Let us recall the following Moser--Trudinger-like inequality (see \cite[Theorem 1.1]{Ruf} and \cite[Theorem 1.7]{LamLu}):
\begin{equation*}
\sup \left\{\int_{\R^{2m}} \left(e^{\alpha_m u^2} - 1\right) \, dx : \|u\|_{H^{m}} \le 1\right\} < +\infty,
\end{equation*}
where $\alpha_m$ is defined in \eqref{alpha_m}. The property below holds as well.

\begin{Lem}\label{lem:prel}
There holds
\[
\sup_{\substack{u \in H^m(\R^{2m})\\ |\nabla^m u|_2^2 \le 1}} \frac{1}{|u|_2^{2}} \int_{\R^{2m}} \left(e^{\alpha u^2} - 1\right) \, dx
\begin{cases}
< +\infty \quad \text{if } \alpha < \alpha_m\\
= +\infty \quad \text{if } \alpha \ge \alpha_m.
\end{cases}
\]
\end{Lem}
\begin{proof} For the proof see \cite[Theorem 1.7]{CLZ}.
\end{proof}

We also present a generalization of \cite[Lemma 1]{do} with $N=2$, which can be proved in a similar way to \cite[Proofs of Theorems 1.4 and 1.1]{LamLu} (see also \cite[Proof of Proposition 2.5]{BMS}).

\begin{Lem}\label{lem:do}
For every $\alpha > 0$ and $u \in H^m(\R^{2m})$, $e^{\alpha u^2}-1 \in L^1(\R^{2m})$.
\end{Lem}

We introduce the norm
\begin{equation}\label{eq:norm}
\|u\|_\mu^2 := \int_{\R^{2m}} |\nabla^{m} u|^2 \, dx + \int_{\R^{2m}} u^2 \, dx + \mu\int_{\R^{2m}} \frac{u^2}{|x|^{2m}} \, dx
\end{equation}
in the space $X^m$, defined in \eqref{def:X}. Additionally, to simplify the notations, we denote
$$
[u]_\mu^2 := \int_{\R^N} |\nabla^{m} u|^2 + \frac{\mu}{|x|^{2m}} u^2 \, dx.
$$

Let us recall the following result from \cite[Proposition 2.5]{BMS}.

\begin{Prop}[Poho\v{z}aev identity]\label{poh:local}
Let $f\colon \R \rightarrow \R$ be a continuous function satisfying
\begin{equation}\label{eq:local}
\hbox{for all }q \ge 2\hbox{ and }\alpha> \alpha_m,\hbox{ there holds }\displaystyle |f(s)| \lesssim |s| + |s|^{q-1}(e^{\alpha s^2}-1).
\end{equation}
Let $u \in X^m_\textup{rad}$ be a weak solution to 
$$
(-\Delta)^m u + \frac{\mu}{|x|^{2m}} u = f(u).
$$
Then
\begin{equation*}
\int_{\R^{2m}} F(u) \, dx = 0,
\end{equation*}
where $F(s) := \int_0^s f(t) \, dt$.
\end{Prop}

Now observe that 
from Proposition \ref{poh:local} we know that every solution $u\in X_\textup{rad}^m$ to \eqref{eq} satisfies the following Poho\v{z}aev identity
$$
\int_{\R^{2m}} \frac{\eta}{4} u^4 + G(u) - \frac{\lambda}{2} u^2 \, dx = 0
$$
provided that $g$ is continuous and satisfies \eqref{eq:local}.
Moreover, every solution to \eqref{eq} also satisfies the Nehari identity
$$
\int_{\R^{2m}} |\nabla^{m} u|^2 \, dx + \int_{\R^{2m}} \frac{\mu}{|y|^{2m}} u^2 \, dx+ \lambda \int_{\R^{2m}} u^2 \, dx - \int_{\R^{2m}} \eta u^4 + g(u)u \, dx = 0.
$$
In particular, every nontrivial solution in $X_\textup{rad}^m$ satisfies $M(u) = 0$, where
\begin{equation}\label{def:M}
M(u) := \int_{\R^{2m}} |\nabla^{m} u|^2 \, dx + \int_{\R^{2m}} \frac{\mu}{|y|^{2m}} u^2 \, dx - \int_{\R^{2m}} \frac{\eta}{2} u^4 + H(u) \, dx,
\end{equation}
i.e., it belongs to $\cM$ -- cf \eqref{def:cM}. Notice that $\cM$ is a $\cC^1$-manifold in $X_\textup{rad}^m$ with $\codim \cM = 1$ under (\ref{A0})--(\ref{A4}), see Lemma \ref{lem:manifold} below.

We recall the Gagliardo--Nirenberg inequality \cite[p. 125]{Nirenberg}, i.e., for every $p \in (2,+\infty)$ there exist $C_{p}>0$ such that
\begin{equation}\label{eq:GN}
|u|_p \leq C_{p}  |\nabla^{m} u|^{1-2/p}_2  |u|_2^{2/p} \quad \text{for every } u \in H^{m}(\R^{2m}),
\end{equation}
and $C_{p}$ is optimal.

Finally, we also recall the following integral characterization of $\preceq$ (cf. \cite[Lemma 2.1]{BM} and \cite[Lemma 3.1]{BMS}).
\begin{Lem}
Suppose that $f_1, f_2 \colon \R\to\R$ are continuous and satisfy
\[
f_j(s) = \cO(s^2) \text{ as } s \to 0 \quad \text{and} \quad \lim_{|s| \to +\infty} f_j(s) / e^{\alpha s^2} = 0 \text{ for all } \alpha > \alpha_m
\]
($j \in \{1,2\}$). Then, $f_1(s) \preceq f_2(s)$ for $s \in \R \setminus \{0\}$ if and only if $f_1 \le f_2$ and
$$
\int_{\R^{2m}} f_1(u) \, dx < \int_{\R^{2m}} f_2(u) \, dx
$$
for every $u \in H^m (\R^{2m})$.
\end{Lem}

\section{The proof}\label{sec:proofs}

In this section, we will consider the following assumption:
\begin{equation}\label{eq:Heta}
\text{There exists } s \ne 0 \text{ such that } \frac{\eta}{2} s^4 + H(s) > 0.
\end{equation}
If $\eta > 0$, this follows, e.g., from (\ref{A4}); if $\eta = 0$, this follows, e.g., from (\ref{A2}) and (\ref{A4}), or (\ref{A3}) and (\ref{A4}). In these cases, we actually have that $\frac{\eta}{2} s^4 + H(s) > 0$ for every $s \ne 0$.

First, we will show that $\cM$ is a differentiable manifold.
\begin{Lem}\label{lem:manifold}
If (\ref{A0}) and \eqref{eq:Heta} hold, then $\cM$ is a $\cC^1$-manifold in $X_\textup{rad}^m$ of codimension 1.
\end{Lem}
\begin{proof}
We begin with proving that $\cM \ne \emptyset$. If $s \in \R$ is such that $\frac{\eta}{2} s^4 + H(s) > 0$, then, for $R \ge 2$ we can build $u \in \cC^\infty(\R^{2m}) \cap X_\textup{rad}^m$ such that $u(x) = s$ if $2 \le |x| \le R$, $u(x) = 0$ if $|x| \ge R+1$ or $|x| \le 1$, and $0 \le u \le 1$. If $R \gg 1$, we compute that $\int_{\R^{2m}} \frac{\eta}{2} u^4 + H(u) \, dx > 0$. Next, defining
\begin{equation}\label{def:r}
r(u) :=  \left( \frac{\displaystyle \int_{\R^{2m}} \frac{\eta}{2} u^4 + H(u) \, dx}{\displaystyle \int_{\R^{2m}} |\nabla^{m} u|^2 + \frac{\mu}{|x|^{2m}} u^2 \, dx} \right)^{1/2},
\end{equation}
we easily check that $u\bigl(r(u)\cdot\bigr) \in \cM$.\\
Now, take any $v \in \cM$ and suppose that $M'(v) = 0$, i.e., $v$ is a weak solution to 
$$
(-\Delta)^{m} v + \frac{\mu}{|x|^{2m}} v = 2 \eta v^3 + \frac12 h(v).
$$
From Proposition \ref{poh:local}, we get that $\int_{\R^{2m}} \frac{\eta}{2} v^4 + H(v) \, dx = 0$, whence $[v]_\mu = 0$ because $v \in \cM$, which is impossible. Thus, $M'(v) \neq 0$, and the proof is complete.
\end{proof}

Throughout this section, we will exploit the property below with $g$ and $h$.

\begin{Rem}\label{rem:growth}
Every continuous function $f \colon \R \to \R$ such that
\[
f(s) = \cO(|s|) \text{ as } s \to 0 \quad \text{and} \quad \lim_{|s| \to +\infty} f(s) / e^{\alpha s^2} = 0 \text{ for all } \alpha > \alpha_m
\]
satisfies \eqref{eq:local}.
\end{Rem}

\begin{Lem}\label{lem0}
If (\ref{A0}), (\ref{A1}), and \eqref{eq:Hstrict} hold, then there exists $\delta>0$ such that
\[
M(u) \gtrsim [u]_\mu^2
\]
for every $u \in \cD \cap H^m(\R^{2m})$ such that $|\nabla^m u|_2 \le \delta$.
\end{Lem}
\begin{proof}
Take $\varepsilon > 0$ and $\alpha > \alpha_m$. From (\ref{A0}), (\ref{A1}), and Remark \ref{rem:growth}, there exists $c_{\alpha,\varepsilon} > 0$ such that for all $s \in \R$
\[
H(s) \le \varepsilon s^4 + c_{\alpha,\varepsilon} s^4 \left(e^{\alpha s^2} - 1\right),
\]
whence, taking $t \in (1,+\infty)$ and using H\"older's inequality,
\begin{align*}
\int_{\R^{2m}} H(u) \, dx & \le \varepsilon |u|_4^4 + c_{\alpha,\varepsilon} \int_{\R^{2m}} u^4 \left(e^{\alpha u^2} - 1\right) \, dx\\
& \le \varepsilon |u|_4^4 + c_{\alpha,\varepsilon} |u|_{4t'}^4 \left(\int_{\R^{2m}} \left(e^{\alpha u^2} - 1\right)^t \, dx\right)^{1/t}\\
& \le \varepsilon |u|_4^4 + c_{\alpha,\varepsilon} |u|_{4t'}^4 \left(\int_{\R^{2m}} e^{\alpha t u^2} - 1 \, dx\right)^{1/t}
\end{align*}
From Lemma \ref{lem:prel} and \eqref{eq:GN}, if $|\nabla^m u|_2 < \delta_1 := \sqrt{\alpha_m/(\alpha t)}$, then there exists $\kappa = \kappa(\delta_1) > 0$ such that
\begin{align*}
\int_{\R^{2m}} H(u) \, dx & \le \varepsilon |u|_4^4 + c_{\alpha,\varepsilon} \kappa \rho^{1/t} |u|_{4t'}^4 \le \varepsilon C_4^4 \rho |\nabla^m u|_2^2 + c_{\alpha,\varepsilon} \kappa C_{4t'}^4 \rho |\nabla^m u|_2^{4-2/t'}\\
& = \left(\varepsilon C_4^4 \rho + c_{\alpha,\varepsilon} \kappa C_{4t'}^4 \rho |\nabla^m u|_2^{2-2/t'}\right) |\nabla^m u|_2^2.
\end{align*}
Now, let $\varepsilon$ satisfy $(\frac\eta2 + \varepsilon) C_4^4 \rho < 1$ and take $\delta \in (0,\delta_1)$ such that $(\frac\eta2 + \varepsilon) C_4^4 \rho + c_{\alpha,\varepsilon} \kappa C_{4t'}^4 \rho \delta^{2-2/t'} < 1$. If $u \in \cD \cap H^m(\R^{2m})$ and $|\nabla^m u|_2 \le \delta$, then
\begin{align*}
M(u) & = [u]_\mu^2 - \int_{\R^{2m}} \frac{\eta}{2} u^4 + H(u) \, dx\\
& \ge \biggl(1 - \Bigl(\frac\eta2 + \varepsilon\Bigr) C_4^4 \rho + c_{\alpha,\varepsilon} \kappa C_{4t'}^4 \rho \delta^{2-2/t'}\biggr) |\nabla^m u|_2^2 + \mu \int_{\R^{2m}} \frac{u^2}{|x|^2} \, dx. \qedhere
\end{align*}
\end{proof}

\begin{Cor}\label{cor:bdda}
If (\ref{A0}), (\ref{A1}), \eqref{eq:Hstrict}, and \eqref{eq:Heta} hold, then
\[
\inf_{u \in \cD \cap \cM} |\nabla^m u|_2 > 0.
\]
\end{Cor}
\begin{proof}
If, by contradiction, there exists $(u_n)_n \in \cD \cap \cM$ such that $\lim_n |\nabla^m u_n|_2 = 0$, then from Lemma \ref{lem0} there follows $0 = M(u_n) \gtrsim [u_n]_\mu^2 > 0$ for $n \gg 1$, which is impossible.
\end{proof}

\begin{Lem}\label{lem:cpt}
Let (\ref{A0}) and (\ref{A1}) hold. If $(u_n) \subset H^m_\textup{rad}(\R^{2m}) \cap \cD$ and $\nu \in (0,1)$ are such that $|\nabla^m u_n|_2 \le \nu$ for every $n$, then there exists $u \in H^m_\textup{rad}(\R^{2m}) \cap \cD$ such that, up to a subsequence, $u_n \rightharpoonup u$ in $H^m(\R^{2m})$, $u_n(x) \to u(x)$ for a.e. $x \in \R^{2m}$, and $G(u_n) \to G(u)$ and $H(u_n) \to H(u)$ in $L^1(\R^{2m})$.
\end{Lem}
\begin{proof}
Since $(u_n)$ is bounded, there exists $u \in H^m_\textup{rad}(\R^{2m}) \cap \cD$ such that, up to a subsequence, $u_n \rightharpoonup u$ in $H^m(\R^{2m})$ and $u_n(x) \to u(x)$ for a.e. $x \in \R^{2m}$. Fix $\alpha \in (\alpha_m,\alpha_m/\nu^2)$ and $t \in \bigl(1,\alpha_m/(\alpha \nu^2)\bigr)$. From the compact embedding $H^m_\textup{rad}(\R^{2m}) \hookrightarrow L^p(\R^{2m})$ for every finite $p > 2$, we have that $u_n \to u$ in $L^4(\R^{2m}) \cap L^{2t'}(\R^{2m})$.
From (\ref{A0}), (\ref{A1}), and Remark \ref{rem:growth}, for every $\varepsilon > 0$ and $\alpha$ as before, there exists $C > 0$ such that for all $s \in \R$,
\begin{equation*}
\max\{|G(s)|,|H(s)|\} \le \varepsilon s^4 + C s^2 \left(e^{\alpha s^2} - 1\right),
\end{equation*}
hence
\begin{equation*}
\max\{|G(u_n)|,|H(u_n)|\} \lesssim u_n^4 + u_n^2 \left(e^{\alpha u_n^2} - 1\right).
\end{equation*}
Moreover, from Lemma \ref{lem:prel} and the fact that $|u_n|_2 \le \sqrt{\rho}$,
\begin{align*}
\int_{\R^{2m}} u_n^2 (e^{\alpha u_n^2} - 1) \, dx & \le |u_n|_{2t'}^2 \left(\int_{\R^{2m}} \left(e^{\alpha u_n^2} - 1\right)^t \, dx\right)^{1/t} \le |u_n|_{2t'}^2 \left(\int_{\R^{2m}} e^{\alpha t u_n^2} - 1 \, dx\right)^{1/t}\\
& \le |u_n|_{2t'}^2 \left(\int_{\R^{2m}} e^{\alpha t M^2 u_n^2 / |\nabla^m u_n|_2^2} - 1 \, dx\right)^{1/t} \lesssim |u_n|_{2t'}^2,
\end{align*}
so the statement follows from a variant of Lebesgue's Dominated Convergence Theorem.
\end{proof}

From now on, we shall denote
\begin{equation*}
c_\beta(\rho) := \inf_{\cD \cap \cM} J.
\end{equation*}

\begin{Lem}\label{lem:u_n-small}
If (\ref{A0}), (\ref{A4}), \eqref{eq:Hstrict}, and \eqref{eq:Heta} hold and $(u_n) \subset \cD \cap \cM$ is such that $\lim_n J(u_n) = c_\beta(\rho)$, then
\[
\limsup_n [u_n]_\mu^2 \le \frac{4}{2 - \eta C_4^4 \rho} \frac{\theta-2}{\theta-4} c_\beta(\rho).
\]
\end{Lem}
\begin{proof}
Since $(u_n) \subset \cM$, from (\ref{A0}) and (\ref{A4}) we have
\begin{align*}
c_\beta(\rho) + o_n(1) & = \frac12 [u_n]_\mu^2 - \int_{\R^{2m}} \frac{\eta}{4} u_n^4 + G(u_n) \, dx = \frac12 \int_{\R^{2m}} g(u_n) u_n - 4 G(u_n) \, dx\\
& \ge \frac{\theta-4}{2} \int_{\R^{2m}} G(u_n) \, dx,
\end{align*}
whence
\begin{align*}
c_\beta(\rho) + o_n(1) = \frac12 [u_n]_\mu^2 - \frac{\eta}{4} |u_n|_4^4 - \int_{\R^{2m}} G(u_n) \, dx \ge \left(\frac12 - \frac{\eta}{4} C_4^4 \rho\right) [u_n]_\mu^2 - \frac{2}{\theta-4} c_\beta(\rho) + o_n(1),
\end{align*}
and we conclude letting $n \to +\infty$.
\end{proof}

\begin{Cor}\label{cor+}
If (\ref{A0}), (\ref{A1}), (\ref{A4}), \eqref{eq:Hstrict}, and \eqref{eq:Heta} hold, then
\[
c_\beta(\rho) > 0.
\]
\end{Cor}
\begin{proof}
It follows from Corollary \ref{cor:bdda} and Lemma \ref{lem:u_n-small}.
\end{proof}

\begin{Lem}\label{lem:c-small}
If (\ref{A0}), (\ref{A2}), \eqref{eq:Hstrict}, and \eqref{eq:Heta} hold, then
\[
c_\beta(\rho) \le \left(\frac12 - \frac{1}{p-2}\right) \left(\frac{1}{\beta (p-2) C_p^p \rho}\right)^{2/(p-4)} \left(1 - \frac{\eta}{2} C_4^4 \rho\right)^{(p-2)/(p-4)}.
\]
\end{Lem}
\begin{proof}
Consider $(u_n) \subset \cD \cap \cM$ such that $\lim_n J(u_n) = c_\beta(\rho)$. From (\ref{A0}), (\ref{A2}), and \eqref{eq:GN},
\begin{align*}
c_\beta(\rho) + o_n(1) & = \frac12 [u_n]_\mu^2 - \int_{\R^{2m}} \frac{\eta}{4} u_n^4 + G(u_n) \, dx \le \frac12 [u_n]_\mu^2 - \frac{\eta}{4} |u_n|_4^4 - \beta |u_n|_p^p\\
& \le \frac12 \left(1 - \frac{\eta}{2} C_4^4 \rho\right) [u_n]_\mu^2 - \beta C_p^p \rho [u_n]_\mu^{p-2}.
\end{align*}
Direct computations show that the maximum of the function
\[
(0,+\infty) \ni s \mapsto \frac12 \left(1 - \frac{\eta}{2} C_4^4 \rho\right) s^2 - \beta C_p^p \rho s^{p-2} \in \R
\]
equals $\left(\frac12 - \frac{1}{p-2}\right) \left(\frac{1}{\beta (p-2) C_p^p \rho}\right)^{2/(p-4)} \left(1 - \frac{\eta}{2} C_4^4 \rho\right)^{(p-2)/(p-4)}$, hence the statement follows letting $n \to +\infty$.
\end{proof}

\begin{Lem}\label{lem:c_0attained}
If (\ref{A0})--(\ref{A2}), (\ref{A4}), \eqref{eq:Hstrict}, and \eqref{eq:beta} hold, then $c_\beta(\rho)$ is attained. If, in addition, $g$ is odd and $m=1$, then $c_\beta(\rho)$ is attained by a nonnegative function in $\cD \cap \cM$.
\end{Lem}

\begin{proof}
Consider $(u_n) \subset \cD \cap \cM$ such that $\lim_n J(u_n) = c_\beta(\rho)$. From Lemmas \ref{lem:u_n-small} and \ref{lem:c-small} and \eqref{eq:beta},
\[
|\nabla^m u_n|_2^2 \le [u_n]_\mu^2 \le \frac{4}{2 - \eta C_4^4 \rho} \frac{\theta-2}{\theta-4} c_\beta(\rho) < 1.
\]
Then, from Lemma \ref{lem:cpt} with $\displaystyle \nu = \frac{4}{2 - \eta C_4^4 \rho} \frac{\theta-2}{\theta-4} c_\beta(\rho)$, there exists $u \in H^m(\R^{2m}) \cap \cD$ such that, up to a subsequence, $u_n \rightharpoonup u$ in $H^m(\R^{2m})$, $u_n(x) \to u(x)$ for a.e. $x \in \R^{2m}$, and $G(u_n) \to G(u)$ and $H(u_n) \to H(u)$ in $L^1(\R^{2m})$. From Corollary \ref{cor:bdda}, the compact embedding $H^m_\textup{rad}(\R^{2m}) \hookrightarrow L^4(\R^{2m})$, and because $(u_n) \in \cM$, there holds
\[
\int_{\R^{2m}} \frac{\eta}{2} u^4 + H(u) \, dx = \lim_n [u_n]_\mu^2 > 0,
\]
hence $u \ne 0$. Let $r = r(u)$ defined by \eqref{def:r}. Since
\[
[u]_\mu^2 \le \lim_n [u_n]_\mu^2 = \int_{\R^{2m}} \frac{\eta}{2} u^4 + H(u) \, dx,
\]
we have that $r \ge 1$, thus $u(r\cdot) \in \cD \cap \cM$ and, from Corollary \ref{cor+}, (\ref{A4}) and Fatou's Lemma,
\begin{align*}
0 < c_\beta(\rho) & \le J\bigl(u(r\cdot)\bigr) = \frac{1}{2r^{2m}} \int_{\R^{2m}} H(u) - 2G(u) \, dx \le \frac12 \int_{\R^{2m}} H(u) - 2 G(u) \, dx\\
& \le \lim_n \frac12 \int_{\R^{2m}} H(u_n) - 2 G(u_n) \, dx = \lim_n J(u_n) = c_\beta(\rho).
\end{align*}
This yields that $r = 1$, $u \in \cD \cap \cM$, and $J(u) = c_\beta(\rho)$. Finally, if $m=1$ and $g$ is odd, then $(u_n)$ can be replaced with $(|u_n|)$.
\end{proof}

For $u\in X^m \setminus\{0\}$, we define 
\begin{equation}\label{vp_u}
\vp_u(s) :=J(s^{m} u(s \cdot)),\quad s\in (0,+\infty),
\end{equation}
and we note the following, easy to check, scaling property.

\begin{Lem}\label{lem:scaling}
Fix $s > 0$ and $u \in X^m \setminus \{0\}$. Let $v := s^m u(s \cdot)$. Then
\begin{enumerate}
	\item $|v|_{2} = |u|_{2}$, 
	\item $[v]_\mu^2 = s^{2m} [u]_\mu^2$.
\end{enumerate}
\end{Lem}

\begin{Lem}\label{lem:noM0}
If (\ref{A0}), (\ref{A3}), and \eqref{eq:Heta} hold, then $\left\{u \in \cM : \vp_u''(1) = 0\right\} = \emptyset$, where $\vp_u$ is given by \eqref{vp_u}.
\end{Lem}
\begin{proof}
If by contradiction there exists $u \in \cM$ with $\vp_u''(1) = 0$, then from the identities $M(u) = \vp_u''(1) = 0$ and Lemma \ref{lem:scaling} we obtain (recall from \eqref{def:M} the definition of $M$)
\[
\int_{\R^N} 4 H(u) - h(u)u \, dx = 0,
\]
in contrast with (\ref{A3}) because $u \ne 0$.
\end{proof}

\begin{proof}[Proof of Theorem \ref{th:main}]
From Lemma \ref{lem:c_0attained}, let $\tu \in \cD \cap \cM$ such that $J(u) = c_\beta(\rho)$; we recall that $u \ge 0$ if $m=1$ and $g$ is odd. Now we prove that, for every $v \in \cS \cap \cM$, the functional $\bigl(\Phi'(v),M'(v)\bigr) \colon X_\textup{rad}^m \to \R^2$ is surjective, where $\Phi(v) := |v|_2^2$ and $M$ is defined in \eqref{def:M}. As a matter of fact, consider the curve $(0,+\infty) \ni s \mapsto w(s) := s^m v(s \cdot) \in \cS$ (see Lemma \ref{lem:scaling}). Defining $\gamma(s) := M\bigl(w(s)\bigr) = s \varphi_v'(s)$, where $\varphi_v$ is given in \eqref{vp_u}, we see from Lemma \ref{lem:noM0} that $M'(v) [w'(1)] = \gamma'(1) \ne 0$, and $w'(1)$ belongs to the tangent space to $\cS$ at $v$. This implies that the system
\[
\begin{cases}
\Phi'(v) [a w'(1) + b v] = b \rho = y\\
M'(v) [a w'(1) + b v] = a M'(v) [w'(1)] + b M'(v) [v] = z
\end{cases}
\]
is solvable in $(a,b)$ for every $(y,z) \in \R^2$, whence the surjectivity. Then, from \cite[Proposition A.1]{MS}, there exist Lagrange multipliers $\nu\in\R$ and $\lambda \ge 0$ such that $\tu \in \cD \cap \cM$ solves
$$
(-\Delta)^m \tu + \mu \frac{\tu}{|x|^{2m}} - \eta \tu^3 - g(\tu)+\lambda \tu +\nu \left( (-\Delta)^m \tu+ \mu \frac{\tu}{|x|^{2m}} - \eta u^3 - \frac12 h(\tu)\right)=0,
$$
that is,
\begin{equation}\label{eq:mu}
(1+\nu) \left( (-\Delta)^m \tu + \mu \frac{\tu}{|x|^{2m}} \right) + \lambda\tu = \eta (1+\nu) \tu^3 + g(\tu) + \frac\nu2 h(\tu).
\end{equation}
If $\nu=-1$, then (\ref{A3}) and (\ref{A4}) yield
\begin{align*}
0 \leq \lambda \int_{\R^{2m}} \tu^2 \, dx & = \frac12 \int_{\R^{2m}} 2 g(\tu)\tu- h(\tu) \tu \, dx < \int_{\R^{2m}} 4 G(\tu)-  g(\tu) \tu \, dx\\
& \le (4 - \theta) \int_{\R^{2m}} G(u) \, dx \le 0,
\end{align*}
which is a contradiction. Consequently, $\tu$ satisfies the Nehari and Poho\v{z}aev identities related to equation \eqref{eq:mu}, whence
\[
(1+\nu) [\tu]_\mu^2 = (1+\nu) \frac{\eta}{2} |\tu|_4^4 + \int_{\R^{2m}} H(\tu) + \frac{\nu}{2} \bigl( h(\tu)\tu - 2 H(\tu) \bigr) \, dx.
\]
This and $\tu \in \cM$ imply
\[
\nu \int_{\R^N} h(\tu)\tu - 4 H(\tu) \, dx = 0,
\]
hence $\nu = 0$ from (\ref{A3}). Therefore, from \eqref{eq:mu}, we obtain that $\tu$ is a week solution to the differential equation in \eqref{eq}. Observe that $\lambda = 0$ if $\tu \in \cD \setminus \cS$, thus the proof is concluded once we have proved that $\lambda > 0$. Suppose by contradiction that $\lambda = 0$. Then $\tu$ satisfies
\[
(-\Delta)^m \tu+\frac{\mu}{|y|^{2m}} \tu = \eta \tu^3 + g(\tu),
\]
and so from Proposition \ref{poh:local} we obtain $\int_{\R^{2m}} \frac{\eta}{4} \tu^4 + G(\tu) \, dx = 0$, in contrast with $G \ge 0$. 
\end{proof}

\section*{Acknowledgements}

BB was partly supported by the National Science Centre, Poland (Grant no. \linebreak 2022/47/D/ST1/00487). OHM was supported in part by CNPq Proc. 303256/2022-2 and FAPESP Projeto Tem\'atico Proc. 2022/16407-1. JS is a member of GNAMPA (INdAM) and was partly supported by the GNAMPA project {\em Metodi variazionali e topologici per alcune equazione di Schr\"odinger nonlineari}. This work was partly supported by the Thematic Research Programme “Variational and geometrical methods in partial differential equations”, University of Warsaw, Excellence Initiative - Research University.

%

\end{document}